\documentclass[11pt]{article}       
\usepackage{geometry}               
\usepackage{graphicx}
\usepackage{caption}
\usepackage{subcaption}
\usepackage{comment}
\usepackage{amsmath} 
\usepackage{amssymb}  
\usepackage{amsthm}  
\usepackage{bm}
\usepackage{lipsum}
\usepackage{color}
\usepackage{multirow}
\usepackage{array}
\usepackage{cite}
\usepackage{hyperref}

\newtheorem{prop}{Proposition}

\newtheorem{cor}{Corollary}
\newtheorem{lem}{Lemma}
\newtheorem{thm}{Theorem}
\newtheorem{rem}{Remark}
\newtheorem{assumption}{Assumption}

\numberwithin{equation}{section}

\newcommand{\R}{\mathbb{R}}

\title{On a minimum eradication time for the SIR model with time-dependent coefficients}

\author{%
  Jiwoong Jang \thanks{Department of Mathematics, University of Wisconsin-Madison.}%
  \and Yeoneung Kim \thanks{Department of Applied Artificial Intelligence, Seoul National University of Science and Technology.}%
  }

\date{}
\providecommand{\keywords}[1]{\textbf{Key words.} #1}

\begin{document}
\maketitle

\pagestyle{myheadings}
\thispagestyle{plain}

\begin{abstract}
We study the minimum eradication time problem for controlled Susceptible-Infected-Recovered (SIR) epidemic models that incorporate vaccination control and time-varying infected and recovery rates. Unlike the SIR model with constant rates, the time-varying model is more delicate as the number of infectious individuals can oscillate, which causes ambiguity for the definition of the eradication time. We accordingly introduce two definitions that describe the minimum eradication time, and we prove that for a suitable choice of the threshold, the two definitions coincide. We also study the well-posedness of time-dependent Hamilton--Jacobi equation that the minimum eradication time satisfies in the viscosity sense and verify that the value function is locally semiconcave under certain conditions.

\end{abstract}


\section{Introduction}
We are interested in studying an eradication time for the controlled Susceptible-Infectious-Recovered (SIR in short henceforth) model with time-varying rates $\beta(t)$ and $\gamma(t)$:
\begin{equation*}
    \begin{cases}
        \dot S &= - \beta(t) S I - \alpha(t) S,\\
        \dot I &= \beta(t) S I - \gamma(t) I,
    \end{cases}
\end{equation*}
where $\beta(t)$ and $\gamma(t)$ denote a time-dependent infected/recovery rate, respectively, and  $\alpha(t)$ represents a vaccination control. The goal of this paper is to study the mathematical properties of the value function that represents the minimum eradication time, defined by the first time at which the population $I$ of infectious is less than or equal to $\mu$ and remains below afterward for a given small threshold $\mu>0$. It turns out that the eradication time should be defined carefully, and its precise definition will be given in Subsection \ref{subsec:notation}. 

For time-independent rates $\beta,\gamma>0$, the minimum eradication time is always well-defined as $I$ shows a simple behavior, either decreasing or increasing first and decreasing afterward. However, for our case, more careful analysis should be carried out as the number of infectious individuals $I$ can oscillate; for instance, even after $I$ goes below a given threshold $\mu>0$, it can bounce up and down several times. 

In this regard, for time-varying rates $\beta(t)$ and $\gamma(t)$, given $\mu_0>0$, the selection of the threshold parameter, denoted as $\mu$, plays a crucial role in accurately identifying the minimum time at which the variable $I$ crosses $\mu$ and remains below this threshold for the duration of the observation as long as $I(0)\geq \mu_0$. More precisely, when $I$ oscillates around the threshold, one can observe the ambiguity and the discontinuity of the eradication time as demonstrated in Figure~\ref{fig:two_eradication}. This paper proves with the compactness argument that given any $\mu_0$, we can select $\mu$ small enough so that the ambiguity and instability of the two types of eradication are avoided as long as $I(0) \geq \mu_0$. Furthermore, we also present the time-dependent Hamilton--Jacobi equation associated with as well as the local semiconcavity result.

\subsection{Literature review}

We introduce a list, but by no means complete, of the works on the vaccination strategy and the eradication time for SIR epidemic models. The SIR model is a classical model as studied in \cite{kermack1927SIR}, and its variants have received a lot of attention particularly during and after the outbreak of COVID-19. The vaccination strategy as a control and the eradication time as a minimum cost function were investigated with optimal control theory \cite{barro2018optimal,pierre-alexandre2007optimal,bolzoni2017time,ev2016optimalcontrol}. For numerical simulations of the eradication for the time-varying SIR model, we refer to \cite{chen2020SIR}. The minimum eradication time problem in the aspect of free end-time optimal control problem was first studied by~\cite{bolzoni2017time} where the authors claim that the optimal plan is to remain inactive and provide the maximum control after a certain point, which is called switching control. 

In~\cite{ev2016optimalcontrol}, various sufficient conditions to ensure the eradication of disease for a time-varying SIR model were provided under some structural assumptions on the dynamics and the transmission rate such as periodicity. Another interesting work related to our paper is~\cite{lv2023time} where the authors study the eradication time for the Susceptible-Exposed-Infected-Susceptible compartmental model under the constraint of resources. In their paper, it was shown that the optimal vaccination control is indeed bang-bang control and there is a trade-off between the minimum eradication time and the total resources under the assumption that all parameters in the model are constants.

For mathematical treatments, the eradication time for controlled SIR models with constant infected and recovery rates was first studied as a viscosity solution to a static first-order Hamilton-Jacobi equation in \cite{hynd2021eradication}. Also, a critical time at which the infected population starts decreasing was analyzed in \cite{hynd2022critical}. The works \cite{hynd2021eradication,hynd2022critical} are for the SIR model with constant rates $\beta$ and $\gamma$.

To the best of our knowledge, this is the first work that studies the minimum eradication time for time-dependent SIR epidemic models in the framework of the dynamic programming principle and viscosity solutions. Also, we observe that for an arbitrarily given threshold, denoted by $\mu$ below, we may not have a unique description of the eradication time in time-varying environments, and we show that for a suitable choice of $\mu$, we necessarily have a unique definition of the eradication time and that this enjoys mathematical properties (such as the continuity and the semiconcavity). For this purpose, we separate the threshold $\mu$ from an initial population $I(0)$ of infectious, and this may suggest that with time-dependent rates, $\mu$ needs to be small enough compared to $I(0)$ for simulations, where the continuity is implicitly assumed.

\subsection{Notations}\label{subsec:notation}
We fix $\mu>0,\ \overline{\beta}\geq\underline{\beta}>0,\ \overline{\gamma}\geq\underline{\gamma}>0$ and continuous functions $\beta:[0,\infty)\to[\underline{\beta},\overline{\beta}]$, $\gamma:[0,\infty)\to[\underline{\gamma},\overline{\gamma}]$ throughout this paper. Let us define the set $\mathcal{A}$ of admissible controls and the data set $\mathcal{D}$ as follows:
\begin{align*}
\mathcal{A}&:=\left\{\alpha\in L^{\infty}([0,\infty)): 0\leq\alpha(t)\leq1 \textrm{ a.e. }t\geq0\right\},\\
\mathcal{D}&:=[0,\infty)\times[\mu,\infty)\times[0,\infty)\times\mathcal{A}.
\end{align*}
The set $\mathcal{A}$ is endowed with the weak$^{\ast}$ topology inherited from that of $L^{\infty}([0,\infty))$. The intervals $[0,\infty),[\mu,\infty),[0,\infty)$ are endowed with their usual topologies, and the data set $\mathcal{D}$ is endowed with their product topology.

For a given datum $d=(x,y,t,\alpha)\in\mathcal{D}$, we define $(S^d,I^d)$ to be the flow of the following ODE:
\begin{equation}\label{eq:translated}
    \begin{cases}
        \dot S^d &= - \beta^{t} S^d I - \alpha^{t} S^d,\\
        \dot I^d &= \beta^{t} S^d I^d - \gamma^{t} I^d,\\
        S^d(0)&=x,\\
        I^d(0)&=y,
    \end{cases}
\end{equation}
where $\alpha^{t}=\alpha(\cdot+t),\ \beta^{t}=\beta(\cdot+t),\ \gamma^{t}=\gamma(\cdot+t)$. By $(S,I)$ the flow associated with a datum $d$, we mean $(S,I)=(S^d,I^d)$ in this paper. When the associated datum $d$ is clear in the context, we abbreviate the superscript $d$ in $(S^d,I^d)$.

For a given datum $d=(x,y,t,\alpha)\in\mathcal{D}$, we define the upper (lower) value functions $\overline{u}^{\alpha}(x,y,t)$ ($\underline{u}^{\alpha}(x,y,t)$), respectively, by
\begin{align*}
\overline{u}^{\alpha}(x,y,t)&:=\sup \{s\geq 0 : I^d(s) \geq \mu\},\\
\underline{u}^{\alpha}(x,y,t)&:=\inf \{s\geq 0: I^d(s+a) \leq \mu,\forall a\geq0\}.
\end{align*}
For $(x,y,t)\in[0,\infty)\times[\mu,\infty)\times[0,\infty)$, we let
\begin{align*}
\overline{u}(x,y,t)&:=\inf_{\alpha\in\mathcal{A}}\overline{u}^{\alpha}(x,y,t),\\
\underline{u}(x,y,t)&:=\inf_{\alpha\in\mathcal{A}}\underline{u}^{\alpha}(x,y,t).
\end{align*}

It turns out that the value functions $\overline{u},\underline{u}$ enjoy the following important properties, which are our main contributions and are stated in the following subsection.

\subsection{Main results}
\begin{thm}\label{thm:sub/supersolution}
The value function $\overline{u}$ ($\underline{u}$, resp.) is upper semicontinuous (lower semicontinuous, resp.) on $[0,\infty)\times[\mu,\infty)\times[0,\infty)$. Moreover, $\overline{u}$ ($\underline{u}$, resp.) is a viscosity subsolution (supersolution, resp.) to
\begin{align}
-\partial_t u+\beta(t)xy\partial_xu+x(\partial_xu)_++(\gamma(t)-\beta(t)x)y\partial_yu=1\label{eqn:main}
\end{align}
in $(0,\infty)\times(\mu,\infty)\times(0,\infty)$. Here, $(\cdot)_+$ denotes the positive part of the argument.
\end{thm}

The functions $\overline{u},\underline{u}$ are natural in this aspect. However, Figure \ref{fig:two_eradication} indicates the discrepancy of $\overline{u}$ and $\underline{u}$, meaning we might not have $\overline{u}=\underline{u}$. This ambiguity is not observed in the time-independent SIR model studied in \cite{hynd2021eradication}.

\begin{figure}[h]
    \centering
    \begin{tabular}{cc}\includegraphics[width=.8\textwidth,clip=false,trim=0.3cm 0.1cm 0.1cm 0.1cm  ]{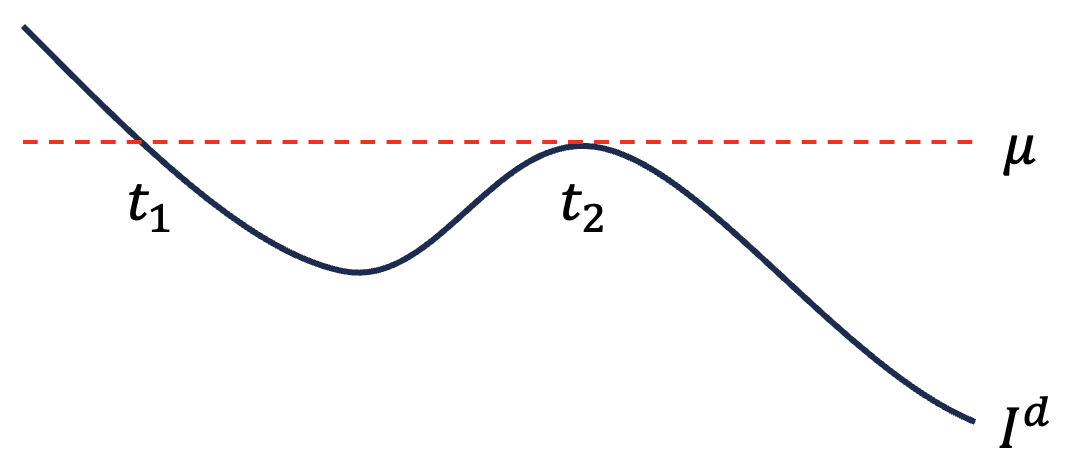}&
    \end{tabular}
    \caption{Two types of eradication time}\label{fig:two_eradication}
\end{figure}

However, we can resolve this ambiguity by taking the following viewpoint; for an initial infected population $I(0)$ that is noticeable, say greater than or equal to $\mu_0$, we require the threshold $\mu$ be much smaller, depending on $\mu_0$ (for instance, at least smaller than $\mu_0$). This perspective aligns with the practical goal of vaccination intervention, aiming to control the spread of disease within the population, especially when controlling the number of infectious individuals under a small threshold.

We start with a fixed $\mu_0>0$. The next result states that for $\mu\in(0,\mu_0]$ small enough, we have $\overline{u}=\underline{u}$, which now becomes a viscosity solution to \eqref{eqn:main} in $(0,\infty)\times(\mu_0,\infty)\times(0,\infty)$. Also, the value function $u:=\overline{u}=\underline{u}$ is characterized by its boundary value conditions. Note that we only assume $\underline{\beta}\leq\beta(t)\leq\overline{\beta}$, $\underline{\gamma}\leq\gamma(t)\leq\overline{\gamma}$ for $t\geq0$, allowing an oscillatory behavior.





\begin{thm}\label{thm:mu}
There exists $\mu_1\in(0,\mu_0]$ depending only on $\mu_0,\beta,\gamma$ such that for every $\mu\in(0.\mu_1]$, it holds that $\overline{u}=\underline{u}$ on $[0,\infty)\times[\mu_0,\infty)\times[0,\infty)$, and therefore, $u:=\overline{u}=\underline{u}$ is a viscosity solution to \eqref{eqn:main} in $(0,\infty)\times(\mu_0,\infty)\times(0,\infty)$. Moreover, if $v$ is a nonnegative viscosity solution to \eqref{eqn:main} and satisfies the boundary conditions
\begin{equation*}
\begin{cases}
v(x,\mu_0,t)=u(x,\mu_0,t) \quad &\text{for}\quad(x,t)\in \left[0,\overline{\gamma}/\underline{\beta}\right]\times [0,\infty),\\
v(0,y,t)=u(0,y,t) \quad &\text{for}\quad (y,t)\in[\mu_0,\infty)\times[0,\infty),\\
v(x,y,0)=u(x,y,0)\quad&\text{for}\quad (x,y)\in\left[0,\overline{\gamma}/\underline{\beta}\right]\times [\mu_0,\infty),
\end{cases}
\end{equation*}
then we have $v=u$.
\end{thm}

\begin{figure}[h]
    \centering
    \begin{tabular}{cc}\includegraphics[width=.8\textwidth,clip=false,trim=0.3cm 0.1cm 0.1cm 0.1cm  ]{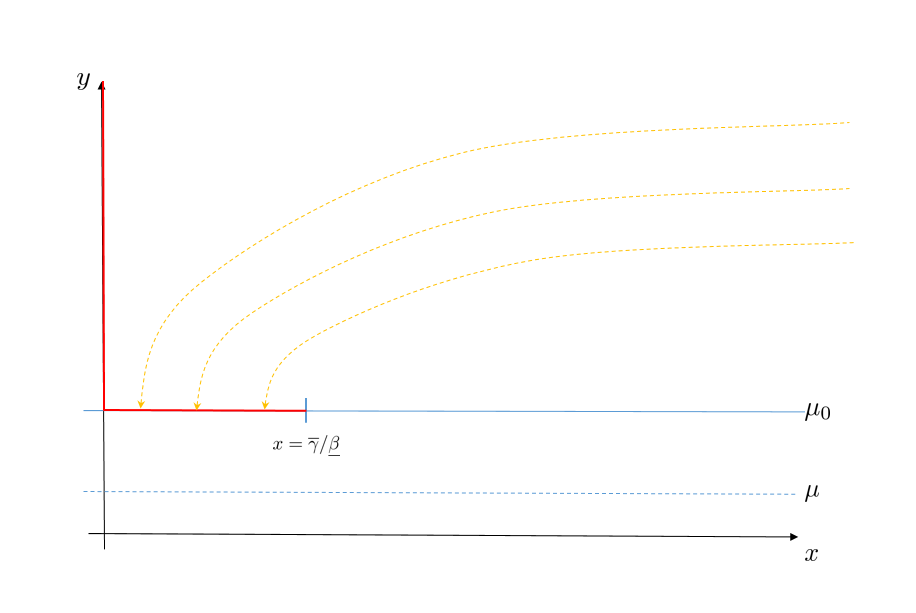}&
    \end{tabular}
    \caption{Effective boundary projected onto the $xy$-plane (the red part)}\label{fig:effectiveboundary}
\end{figure}

Although \eqref{eqn:main} has $-\partial_tu$ instead of $\partial_tu$, the notion of viscosity solutions to \eqref{eqn:main} is the same as the usual one to forward Cauchy problems, for which we refer to \cite[Chapter 1]{Tran2021HJtheory}.

It is worth noting that the Hamiltonian
$$
H(t,x,y,p,q)=\beta(t)xyp+xp_++(\gamma(t)-\beta(t)x)yq
$$
is positively homogeneous of degree 1, and thus, \eqref{eqn:main} has a hidden underlying front propagation structure \cite{Tran2021HJtheory}. Furthermore, only a part of the boundary is needed for the uniqueness result in Theorem \ref{thm:mu}. The front propagation nature and the boundary condition deserve further study.

We state a further regularity property when the transmission rate $\beta$ and the recovery rate $\gamma$ become constant in a small time.

\begin{thm}\label{thm:semiconcavity}
Let $\mu\in(0,\mu_0]$ be chosen as in Theorem \ref{thm:mu} so that $\overline{u}=\underline{u}(=u)$. Suppose that $\beta(t)\equiv\beta_0$, $\gamma(t)\equiv\gamma_0$ for all $t\geq T:=\frac{1}{2\overline{\gamma}}\log\left(\frac{\mu_0}{\mu}\right)$ for some constants $\beta_0,\gamma_0$. Then, $u(x,y,t)$ is locally semiconcave in $(0,\infty)\times(\mu_0,\infty)\times(0,\infty)$.
\end{thm}

\subsection*{Organization of the paper.}
The paper is organized as follows. In Section \ref{sec:sub/super}, we review the basic results of the flow \eqref{eq:translated} and prove Theorem \ref{thm:sub/supersolution}. Section \ref{sec:solution} is entirely devoted to the proof of the existence of $\mu\in(0,\mu_0]$ satisfying $\overline{u}=\underline{u}$ on $[0,\infty)\times[\mu_0,\infty)\times[0,\infty)$. In Section \ref{sec:furtherproperties}, we complete the proof of Theorem \ref{thm:mu} by verifying the uniqueness (Theorem \ref{prop:uniqueness}), and we also prove Theorem \ref{thm:semiconcavity}.

\section{Properties of $\overline{u}$ and $\underline{u}$}\label{sec:sub/super}

In this section, we go over the basic properties of the flow of \eqref{eq:translated}. Then, we investigate the semicontinuity of the value functions $\overline{u},\ \underline{u}$. Finally, we check the dynamic programming principle and viscosity sub/supersolution tests. The main reference is \cite{hynd2021eradication}, and we skip similar proofs. The properties coming from the split of $\overline{u}$ and $\underline{u}$ will be explained.

\subsection{Flow of \eqref{eq:translated}}
\begin{lem}\label{lem:finitespeed}
For any $d=(x,y,t,\alpha)\in\mathcal{D}$, there is a unique flow $(S,I)$ associated with $d$. The flow $(S,I)$ is Lipschitz continuous. Namely, we have
\begin{align*}
|\dot{S}|,|\dot{I}|\leq \overline{\beta}(x+y)^2+\max\{1,\overline{\gamma}\}(x+y).
\end{align*}
\end{lem}

\begin{lem}\label{lem:limitzero}
For any $d=(x,y,t,\alpha)\in\mathcal{D}$, the associated flow $(S,I)$ satisfies $\lim_{t\to\infty}I(t)=0$.
\end{lem}

\begin{prop}\label{prop:stability}
Let $d_k=(x_k,y_k,t_k,\alpha_k)\in\mathcal{D}$ and $(S_k,I_k)$ be the associated flow for each $k=0,1,2,\cdots$. Suppose that $d_k\to d_0$ as $k\rightarrow\infty$ in $\mathcal{D}$. Then, $(S_k(t),I_k(t))\rightarrow(S_0(t),I_0(t))$ as $k\rightarrow\infty$ locally uniformly in $t\in[0,\infty)$.
\end{prop}

Now, we state the existence of an optimal control associated with the value function $\underline{u}$. As this property is expected for $\overline{u}$, we give a proof.

\begin{prop}\label{prop:optimalcontrol}
For any $(x,y,t)\in[0,\infty)\times[\mu,\infty)\times[0,\infty)$, there exists $\alpha^{\ast}\in\mathcal{A}$ such that
$$
\underline{u}^{\alpha^{\ast}}(x,y,t)\leq\underline{u}^{\alpha}(x,y,t)
$$
for any $\alpha\in\mathcal{A}$.
\end{prop}
\begin{proof}
Choose a sequence $\{\alpha_k\}_{k=1,2,\cdots}$ in $\mathcal{A}$ such that $\inf_{\alpha\in\mathcal{A}}\underline{u}^{\alpha}(x,y,t)=\lim_{k\to\infty}\underline{u}^{\alpha_k}(x,y,t)$. As the space $\mathcal{A}$ is (sequentially) weak$^{\ast}$ compact, there is a subsequence $\{\alpha_{k_j}\}_{j=1,2,\cdots}$ of $\{\alpha_k\}_{k=1,2,\cdots}$ such that $\alpha_{k_j}\rightarrow\alpha_0$ weak$^{\ast}$ for some $\alpha_0\in\mathcal{A}$ as $j\to\infty$.

For each $j=0,1,2,\cdots$, let $(S_j,I_j)$ be the flow associated with the datum $(x,y,t,\alpha_{k_j})\in\mathcal{D}$ with $k_0:=0$ Let $a\geq0$. Then, by the definition of $\underline{u}^{\alpha}$ with general control $\alpha\in\mathcal{A}$, we have $I_j(\underline{u}^{\alpha_{k_j}}+a)\leq\mu$ for all $j=1,2,\cdots$. Taking the limit $j\to\infty$, we obtain that $I_0(\inf_{\alpha\in\mathcal{A}}\underline{u}^{\alpha}+a)\leq\mu$. Since $a\in[0,\infty)$ was arbitrary, we conclude $\underline{u}^{\alpha_0}(x,y,t)\leq\inf_{\alpha\in\mathcal{A}}\underline{u}^{\alpha}(x,y,t)$ from the definition of $\underline{u}^{\alpha_0}(x,y,t)$.
\end{proof}

\subsection{Semicontinuity}

The following states the semicontinuity of $\underline{u}^{\alpha}(x,y,t),\overline{u}^{\alpha}(x,y,t)$.

\begin{lem}
Let $d_k=(x_k,y_k,t_k,\alpha_k)\in\mathcal{D}$ for $k=1,2,\cdots$. Suppose that $d_k\to d$ as $k\rightarrow\infty$ in $\mathcal{D}$ for some $d=(x,y,t,\alpha)\in\mathcal{D}$. Then,
$$
\underline{u}^{\alpha}(x,y,t)\leq\liminf_{k\to\infty}\underline{u}^{\alpha_k}(x_k,y_k,t_k)
$$
and
$$
\limsup_{k\to\infty}\overline{u}^{\alpha_k}(x_k,y_k,t_k)\leq\overline{u}^{\alpha}(x,y,t).
$$
\end{lem}

We skip the proof, as it is identical to that of \cite[Corollary 3.2]{hynd2021eradication}. Now, we prove the semicontinuity of $\underline{u}(x,y,t),\overline{u}(x,y,t)$.

\begin{prop}
The value function $\underline{u}$ is lower semicontinuous on $[0,\infty)\times[\mu,\infty)\times[0,\infty)$. Also, the value function $\overline{u}$ is upper semicontinuous on $[0,\infty)\times[\mu,\infty)\times[0,\infty)$.
\end{prop}
\begin{proof}
Say $(x_k,y_k,t_k)\rightarrow(x,y,t)$ as $k\rightarrow\infty$ in $[0,\infty)\times[\mu,\infty)\times[0,\infty)$. For each $k=1,2,\cdots,$ choose $\alpha_k\in\mathcal{A}$ such that $\underline{u}(x_k,y_k,t_k)=\underline{u}^{\alpha_k}(x_k,y_k,t_k)$, which is possible due to Proposition \ref{prop:optimalcontrol}. Select a subsequence $\{\alpha_{k_j}\}_{j}$ such that $\liminf_{k\to\infty}\underline{u}(x_k,y_k,t_k)=\lim_{j\to\infty}\underline{u}^{\alpha_{k_j}}(x_{k_j},y_{k_j},t_{k_j})$ and that $\alpha_{k_j}\rightarrow\alpha$ weak$^{\ast}$ as $j\rightarrow\infty$ for some $\alpha\in\mathcal{A}$. Then, by the above lemma, we get
$$
\liminf_{k\to\infty}\underline{u}(x_k,y_k,t_k)=\lim_{j\to\infty}\underline{u}^{\alpha_{k_j}}(x_{k_j},y_{k_j},t_{k_j})\geq\underline{u}^{\alpha}(x,y,t)\geq\underline{u}(x,y,t). 
$$

Let $\delta>0$ be given. Then, we can choose $\alpha\in\mathcal{A}$ such that $\overline{u}^{\alpha}(x,y,t)<\overline{u}(x,y,t)+\delta$. For a sequence $(x_k,y_k,t_k)$ that converges to $(x,y,t)$ in $[0,\infty)\times[\mu,\infty)\times[0,\infty)$, we have
$$
\limsup_{k\to\infty}\overline{u}(x_k,y_k,t_k)\leq\limsup_{k\to\infty}\overline{u}^{\alpha}(x_k,y_k,t_k)\leq\overline{u}^{\alpha}(x,y,t)<\overline{u}(x,y,t)+\delta.
$$
Here we used the above lemma. Letting $\delta\rightarrow0$ gives the upper semicontinuity.
\end{proof}


\subsection{Dynamic programming principle and a viscosity sub/supersolution}
For $d=(x,y,s,\alpha)\in\mathcal{D}$, we have the dynamic programming principle: for $t\in[0,\inf\{t_1\geq0:I^d(t_1)=\mu\}]$, we have
\begin{align*}
\overline{u}^{\alpha}(x,y,s)&=t+\overline{u}^{\alpha}(S^d(t),I^d(t),t+s),\\
\underline{u}^{\alpha}(x,y,s)&=t+\underline{u}^{\alpha}(S^d(t),I^d(t),t+s).
\end{align*}
This is also true for $\overline{u},\underline{u}$, as stated in the next proposition.

\begin{prop}\label{prop:dpp}
Let $x\geq0,y\geq\mu,s\geq0$. If $t\in[0,\inf_{\alpha\in\mathcal{A}}\inf\{t_1\geq0:I^d(t_1)=\mu\}]$ with $d=(x,y,s,\alpha)$,
\begin{align*}
\overline{u}(x,y,s)&=\inf_{\alpha\in\mathcal{A}}\left\{t+\overline{u}(S^d(t),I^d(t),t+s)\right\},\\
\underline{u}(x,y,s)&=\inf_{\alpha\in\mathcal{A}}\left\{t+\underline{u}(S^d(t),I^d(t),t+s)\right\}.
\end{align*}
Moreover, for any control $\alpha^{\ast}\in\mathcal{A}$ such that $\underline{u}(x,y,s)=\underline{u}^{\alpha^{\ast}}(x,y,s)$, we have
$$
\underline{u}(x,y,s)=t+\underline{u}(S^{\ast}(t),I^{\ast}(t),t+s)
$$
and
$$
\underline{u}(S^{\ast}(t),I^{\ast}(t),t+s)=\underline{u}^{\alpha^{\ast}}(S^{\ast}(t),I^{\ast}(t),t+s).
$$
for $t\in[0,\inf\{t_1\geq0:I^{\ast}(t_1)=\mu\}]$. Here, the flow $(S^{\ast},I^{\ast})$ is associated with $(x,y,s,\alpha^{\ast})\in\mathcal{D}$.
\end{prop}

We omit the proof as it is basically the same as that of \cite[Proposition 3.4]{hynd2021eradication}. As a corollary from the dynamic programming principle, we see that the value functions $\overline{u}$ and $\underline{u}$ are a viscosity sub and supersolution, respectively. Once we have the dynamic programming principle, we are able to verify naturally that $\overline{u}$ ($\underline{u}$, resp.) is a viscosity subsolution (supersolution, resp.) to \eqref{eqn:main} (see \cite{Evans2010textbook,Tran2021HJtheory}).

\begin{cor}[Theorem \ref{thm:sub/supersolution}]
The value function $\overline{u}$ ($\underline{u}$, resp.) is a viscosity subsolution (supersolution, resp.) to
$$
-\partial_t u+\beta(t)xy\partial_xu+x(\partial_xu)_++(\gamma(t)-\beta(t)x)y\partial_yu=1
$$
in $(0,\infty)\times(\mu,\infty)\times(0,\infty)$.
\end{cor}
\begin{proof}
Fix the ball $B_{\delta}(x_0,y_0,t_0)\subset(0,\infty)\times(\mu,\infty)\times(0,\infty)$ with a center $(x_0,y_0,t_0)$ and a radius $\delta>0$. Let $\phi$ be a $C^1$ function defined in $B_{\delta}(x_0,y_0,t_0)$ such that $\overline{u}-\phi$ attains a maximum at $(x_0,y_0,t_0)$ in $B_{\delta}(x_0,y_0,t_0)$. Let $a\in[0,1]$ and $\alpha\equiv a\in\mathcal{A}$. Let $(S,I)$ be the flow associated with $(x_0,y_0,t_0,\alpha)\in\mathcal{D}$. As $y_0>\mu$, we have $\overline{u}(x_0,y_0,t_0)>0$. As $(S,I)$ is continuous in time, there exists $t_1>0$ depending also on $\delta>0$ such that $(S(t),I(t))\in B_{\delta}(x_0,y_0,t_0)$ for all $t\in[0,t_1]$.

From
$$
(\overline{u}-\phi)(S(t),I(t),t_0+t)\leq(\overline{u}-\phi)(x_0,y_0,t_0)
$$
for $t\in[0,\min\{\overline{u}(x_0,y_0,t_0),t_1\}]$, we obtain
\begin{align*}
-t&\leq\overline{u}(S(t),I(t),t_0+t)-\overline{u}(x_0,y_0,t_0)\\
&\leq\phi(S(t),I(t),t_0+t)-\phi(x_0,y_0,t_0).
\end{align*}
Here, we used Proposition \ref{prop:dpp} in the first line. This yields
\begin{align*}
-1&\leq\frac{d}{dt}\phi(S(t),I(t),t+t_0)|\Big|_{t=0}\\
&=\partial_t \phi(x_0,y_0,t_0)-(\beta(t_0)x_0y_0+ax_0)\partial_x\phi(x_0,y_0,t_0)-(\gamma(t_0)-\beta(t_0)x_0)y_0\partial_y\phi(x_0,y_0,t_0).
\end{align*}
Taking the supremum over $a\in[0,1]$, we obtain
$$
-\partial_t \phi(x_0,y_0,t_0)+\beta(t_0)x_0y_0\partial_x\phi(x_0,y_0,t_0)+x_0(\partial_x\phi(x_0,y_0,t_0))_++(\gamma(t_0)-\beta(t_0)x_0)y_0\partial_y\phi(x_0,y_0,t_0)\leq1.
$$

Now, let $\psi$ be a $C^1$ function defined in $B_{\delta}(x_0,y_0,t_0)$ such that $\underline{u}-\psi$ attains a minimum at $(x_0,y_0,t_0)$ in $B_{\delta}(x_0,y_0,t_0)$. Take $\alpha^{\ast}\in\mathcal{A}$ such that $\underline{u}(x_0,y_0,t_0)=\underline{u}^{\alpha^{\ast}}(x_0,y_0,t_0)$. Let $(S,I)$ be the flow associated with $(x_0,y_0,t_0,\alpha^{\ast})\in\mathcal{D}$. As $y_0>\mu$, we have $\underline{u}(x_0,y_0,t_0)>0$. As $(S,I)$ is continuous in time, there exists $t_1>0$ depending also on $\delta>0$ such that $(S(t),I(t))\in B_{\delta}(x_0,y_0,t_0)$ for all $t\in[0,t_1]$.

By Proposition \ref{prop:dpp}, we have
$$
\underline{u}(x_0,y_0,t_0)=t+\underline{u}(S(t),I(t),t+t_0)
$$
for $t\in[0,\underline{u}(x_0,y_0,t_0)]$. From
$$
(\underline{u}-\psi)(S(t),I(t),t_0+t)\geq(\underline{u}-\psi)(x_0,y_0,t_0)
$$
for $t\in[0,\min\{\underline{u}(x_0,y_0,t_0),t_1\}]$, we obtain
\begin{align*}
-t&=\underline{u}(S(t),I(t),t_0+t)-\underline{u}(x_0,y_0,t_0)\\
&\geq\psi(S(t),I(t),t_0+t)-\psi(x_0,y_0,t_0).
\end{align*}
Consequently, the fact that $\alpha^{\ast}(s)\in[0,1]$ for almost every $s\geq0$ yields that, for $t\in[0,\min\{\underline{u}(x_0,y_0,t_0),t_1\}]$,
\begin{align*}
-1&\geq\frac{1}{t}\int_0^t\frac{d}{ds}\psi(S(s),I(s),s+t_0)ds\\
&=\frac{1}{t}\int_0^t\left(\partial_t\psi(S(s),I(s),s+t_0)+(-\beta(s+t_0)S(s)I(s)-\alpha^{\ast}(s+t_0)S(s))\partial_x\psi(S(s),I(s),s+t_0)\right.\\
&\qquad\qquad\qquad\qquad\qquad\qquad\qquad\left.+(\beta(s+t_0)S(s)I(s)-\gamma(s+t_0)I(s))\partial_y\psi(S(s),I(s),s+t_0)\right)ds\\
&\geq\frac{1}{t}\int_0^t\left(\partial_t\psi(S(s),I(s),s+t_0)-\beta(s+t_0)S(s)I(s)\partial_x\psi(S(s),I(s),s+t_0)\right.\\
&\qquad\left.-S(s)(\partial_x\psi(S(s),I(s),s+t_0))_++(\beta(s+t_0)S(s)I(s)-\gamma(s+t_0)I(s))\partial_y\psi(S(s),I(s),s+t_0)\right)ds.
\end{align*}
Sending $t\to0^+$ and rearranging the terms gives
$$
-\partial_t \psi(x_0,y_0,t_0)+\beta(t_0)x_0y_0\partial_x\psi(x_0,y_0,t_0)+x_0(\partial_x\psi(x_0,y_0,t_0))_++(\gamma(t_0)-\beta(t_0)x_0)y_0\partial_y\psi(x_0,y_0,t_0)\geq1.
$$
\end{proof}

\section{A viscosity solution $u=\overline{u}=\underline{u}$ from the choice of $\mu>0$}\label{sec:solution}
This section is devoted to the proof of the following theorem.

\begin{thm}\label{thm:choiceofmu}
There exists $\mu_1\in(0,\mu_0]$ depending only on $\mu_0,\beta,\gamma$ such that for every $\mu\in(0,\mu_1]$, we have $\overline{u}=\underline{u}$ on $[0,\infty)\times[\mu_0,\infty)\times[0,\infty)$.
\end{thm}

The theorem means that we can avoid the splitting of the two eradication times by choosing $\mu\in(0,\mu_0]$ small enough that we obtain a viscosity solution $u:=\overline{u}=\underline{u}$ to \eqref{eqn:main} in $(0,\infty)\times(\mu_0,\infty)\times(0,\infty)$.

\begin{lem}
Let $(S,I)$ be the flow of \eqref{eq:translated} associated with a datum $d=(x_0,y_0,t_0,\alpha)\in\mathcal{D}$. If $\underline{u}^{\alpha}(x_0,y_0,t_0)<\overline{u}^{\alpha}(x_0,y_0,t_0)$, then $\dot{I}(s)=0$ for $s=\overline{u}^{\alpha}(x_0,y_0,t_0)$.
\end{lem}
\begin{proof}
The proof of this lemma is straightforward; if $\underline{u}^{\alpha}(x_0,y_0,t_0)<\overline{u}^{\alpha}(x_0,y_0,t_0)$, then there exists $\delta\in(0,\overline{u}^{\alpha}(x_0,y_0,t_0)-\underline{u}^{\alpha}(x_0,y_0,t_0))$ such that $I(s)\leq\mu$ for all $s\in[\overline{u}^{\alpha}(x_0,y_0,t_0)-\delta,\overline{u}^{\alpha}(x_0,y_0,t_0)+\delta]$. Since $I(s)=\mu$ for $s=\overline{u}^{\alpha}(x_0,y_0,t_0)$, we get $\dot{I}(s)=0$ for $\overline{u}^{\alpha}(x_0,y_0,t_0)$.
\end{proof}

Therefore, once we prove the following proposition, then we obtain Theorem \ref{thm:choiceofmu}.

\begin{prop}\label{prop:mu}
For any $\mu_0>0$, there is $\mu_1\in(0,\mu_0]$ depending only on $\mu_0,\beta,\gamma$ such that for every $d=(x_0,y_0,t_0,\alpha)\in\mathcal{D}$ with $y\geq\mu_0$, it holds that $\{I(s): \dot{I}(s)=0,\ s\geq0 \}\subset(\mu_1,\infty),$ where $(S,I)$ is the flow of \eqref{eq:translated} associated with $d$.
\end{prop}

The rest of this section is devoted to the proof of this proposition.

\begin{proof}
Fix $d=(x_0,y_0,t_0,\alpha)\in\mathcal{D}$ with $y\geq\mu_0$. We divide the proof into three steps.

\medskip

\textbf{Step 1}: Case reductions.

\medskip


First of all, for the case $t_0>0$, proving the conclusion for $d=(x_0,y_0,t_0,\alpha)$ is the same as proving for $(x_0,y_0,0,\alpha^{t_0})$ with the changed coefficients $\beta^{t_0},\gamma^{t_0}$. Thus, it suffices to show the case when $t_0=0$.

When $x_0\leq\underline{\gamma}/\overline{\beta}$, we can take $\mu_1=\mu_0$, as $I$ is strictly decreasing when this is the case, and from now on, we may assume $x_0\geq\underline{\gamma}/\overline{\beta},t_0=0$ without loss of generality. 

If $x_0>\overline{\gamma}/\underline{\beta}$, then we can find the minimal time $t_1>0$ such that $S(t_1)=\overline{\gamma}/\underline{\beta}$. The derivative vanishing $\dot{I}(t)=0$ happens only for $t\geq t_1$, and therefore, it suffices to prove the proposition for $(S(t_1),I(t_1),0,\alpha^{t_1})$ with the changed coefficients $\beta^{t_1},\gamma^{t_1}$, which belongs to the case when $x_0\in\left[\underline{\gamma}/\overline{\beta},\overline{\gamma}/\underline{\beta}\right],t_0=0$. From now on, we may assume $x_0\in\left[\underline{\gamma}/\overline{\beta},\overline{\gamma}/\underline{\beta}\right],t_0=0$ without loss of generality.

If $y_0>\mu_0$, then we can find the minimal time $t_2>0$ such that $I(t_2)=\mu_0$ since $\lim_{t\to\infty}I(t)=0$. As $\dot{I}(t_2)\leq0$, we necessarily have $\beta(t_2)S(t_2)-\gamma(t_2)\leq0$, which implies $S(t_2)\leq\overline{\gamma}/\underline{\beta}$. It may happen that $\dot{I}(t)=0$ for some $t\in[0,t_2)$, but $I(t)>\mu_0$ for such $t$. Thus, as long as we require $\mu_1\in(0,\mu_0]$, the case $y_0>\mu_0$ is reduced to the case for $(S(t_2),\mu_0,0,\alpha^{t_2})$ (with the changed coefficients $\beta^{t_2},\gamma^{t_2}$). If $S(t_2)\leq\underline{\gamma}/\overline{\beta}$, we can take $\mu_1=\mu_0$. If not, then $S(t_2)\in\left[\underline{\gamma}/\overline{\beta},\overline{\gamma}/\underline{\beta}\right]$, which belongs to the case $x_0\in\left[\underline{\gamma}/\overline{\beta},\overline{\gamma}/\underline{\beta}\right],y_0=\mu_0,t_0=0$.

From now on, we may assume that $x_0\in\left[\underline{\gamma}/\overline{\beta},\overline{\gamma}/\underline{\beta}\right],y_0=\mu_0,t_0=0$ without loss of generality. We also now fix $\alpha\in\mathcal{A}$.

\medskip

\textbf{Step 2}: Definition of a mapping $T$.

\medskip

Let $X:=\left[\underline{\gamma}/\overline{\beta},\overline{\gamma}/\underline{\beta}\right]$. Define the mapping $T$ as follows:
\begin{align*}
T\ :\ X\ &\longrightarrow\ [0,\infty)\\
x_0\ &\longmapsto\ T(x_0) 
\end{align*}
where $T(x_0):=\inf\left\{t\geq0:S(t)\leq\underline{\gamma}/\overline{\beta}\right\}$, $(S,I)$ is the flow associated with $(x_0,\mu_0,0,\alpha)\in\mathcal{D}$. By the definition of $T$, we have $S(T)=\underline{\gamma}/\overline{\beta}$.

If we can prove that $T$ is bounded, say by $M>0$, then we are done the proof of the proposition. This is because for any $x_0\in X$, it holds that
\begin{align*}
\dot{I}=\beta SI- \gamma I\geq -\overline{\gamma}I\ \implies\ I(t)\geq\mu_0 e^{-\overline{\gamma} M} \textrm{ for any }t\in[0,M].
\end{align*}
The derivative vanishing $\dot{I}(t)=0$ happens only if $S(t)\in\left[\underline{\gamma}/\overline{\beta},\overline{\gamma}/\underline{\beta}\right]$, and this occurs only when $t\leq M$. Consequently, taking $\mu_1=\frac{1}{2}\mu_0 e^{-\overline{\gamma} M}$ completes the proof.

Now, since $X$ is compact, it suffices to show the continuity of $T$.

\medskip

\textbf{Step 3}: Continuity of the mapping $T$.

\medskip

In this step, we prove that if $x_k\rightarrow x$ as $k\to\infty$ in $X$, then $T_k:=T(x_k)\rightarrow T_0:=T(x)$ as $k\to\infty$.

Let $(S_k,I_k)$ be the flow associated with $(x_k,\mu_0,0,\alpha)$ for each $k=1,2,\cdots$, and let $(S_0,I_0)$ be the flow associated with $(x,\mu_0,0,\alpha)$. Then, by Proposition \ref{prop:stability}, $(S_k,I_k)$ converges to $(S_0,I_0)$ locally uniformly in $[0,\infty)$ as $k\to\infty$.

We first check that $T_k\leq T_0+1$ for all but finitely many $k$. If not, then there would be a subsequence $\left\{T_{k_j}\right\}_{j}$ such that
$$
\underline{\gamma}/\overline{\beta}\leq S_{k_j}(T_{k_j})\leq S_{k_j}(T_0+1).
$$
Letting $j\to\infty$ gives
$$
\underline{\gamma}/\overline{\beta}\leq S_0(T_0+1)< S_0(T_0)=\underline{\gamma}/\overline{\beta},
$$
which is a contradiction.

Let
\begin{align*}
m_S&:=\inf_{k\geq1,t\in[0,T_0+1]}\left\{S_k(t)\right\},\\
m_I&:=\inf_{k\geq1,t\in[0,T_0+1]}\left\{I_k(t)\right\}.
\end{align*}
Since $(S_k,I_k)$ converges to $(S_0,I_0)$ uniformly in $[0,T_0+1]$ as $k\to\infty$, we have $m_S,m_I>0$.

Now, for $k\geq1$,
$$
\dot{S}_k=-\beta S_kI_k-\alpha_kS_k\leq-\underline{\beta}m_Sm_I<0,
$$
on $[0,T_0+1]$. This implies, by the mean value theorem, that for $k\geq1$ large enough,
$$
\left|\frac{S_k(T_k)-S_k(T_0)}{T_k-T_0}\right|\geq\underline{\beta}m_Sm_I>0,
$$
unless $T_k=T_0$. Therefore,
$$
\left|T_k-T_0\right|\leq\frac{1}{\underline{\beta}m_Sm_I}\left|S_k(T_k)-S_k(T_0)\right|=\frac{1}{\underline{\beta}m_Sm_I}\left|\underline{\gamma}/\overline{\beta}-S_k(T_0)\right|.
$$
As $S_k(T_0)\to S_0(T_0)=\underline{\gamma}/\overline{\beta}$, we see that $T_k\to T_0$ as $k\to\infty$.

Therefore, $T$ is continuous, and this finishes the proof.
\end{proof}

\begin{rem}
The argument is essentially the proof of the inverse function theorem, as $T$ is the inverse function of $S$. The continuity of $S$ yields that of $T$. We necessarily separate the derivative $\dot{S}$ from 0, which corresponds to the nonzero determinant assumption of the inverse function theorem.
\end{rem}

\section{Further properties of $u=\overline{u}=\underline{u}$}\label{sec:furtherproperties}
\subsection{Uniqueness}
From now on, we assume that $\mu_1\in(0,\mu_0]$ is chosen as in Theorem \ref{thm:choiceofmu} and let $\mu\in(0,\mu_1]$ so that we have a viscosity solution $\overline{u}=\underline{u}$ to \eqref{eqn:main} in $(0,\infty)\times(\mu_0,\infty)\times(0,\infty)$. In this section, we discuss the uniqueness of the solution under prescribed boundary values with a boundedness assumption.

Let us denote $\Omega:=(0,\infty)\times (\mu_0, \infty) \subset \R^2$. Let $u(x,y,t) \in C(\overline {\Omega}\times [0,\infty))$ be a viscosity solution, bounded from below, to
\begin{equation}\label{eq:hj}
\begin{cases}
-\partial_t u+ \beta(t) x y \partial_x u+ x(\partial_x u)_++(\gamma(t) -\beta(t) x ) y\partial_y u=1 \quad&\text{in}\quad \Omega \times (0,\infty),\\
u(x,\mu_0,t)=f(x,t) \quad &\text{on} \quad \left[0,\overline{\gamma}/\underline{\beta}\right]\times [0,\infty),\\
u(0,y,t)=g(y,t) \quad &\text{on}\quad [\mu_0,\infty)\times[0,\infty),\\
u(x,y,0)=h(x,y)\quad&\text{on}\quad \left[0,\overline{\gamma}/\underline{\beta}\right]\times [\mu_0,\infty),
\end{cases}
\end{equation}
where $f,g,h\geq0$ are given continuous functions bounded from below on $\left[0,\overline{\gamma}/\underline{\beta}\right]\times [0,\infty),[\mu_0,\infty)\times[0,\infty),\left[0,\overline{\gamma}/\underline{\beta}\right]\times [\mu_0,\infty)$, respectively.
Then, one sees that $v:=e^{-u}\in C(\overline{\Omega}\times [0,\infty))$ is a bounded viscosity solution to
\begin{equation}\label{eq:hjb}
\begin{cases}
-\partial_t  v+v+ \beta(t) x y \partial_x v - x(\partial_x  v)_- +(\gamma(t) -\beta(t) x ) y\partial_y  v=0 \quad&\text{in}\quad \Omega\times (0,\infty),\\
 v(x,\mu_0,t)=e^{-f(x,t)} \quad &\text{on} \quad \left[0,\overline{\gamma}/\underline{\beta}\right]\times [0,\infty),\\
 v(0,y,t)=e^{-g(y,t)} \quad &\text{on}\quad [\mu_0,\infty)\times[0,\infty),\\
 v(x,y,0)=e^{-h(x,y)} \quad&\text{on}\quad \left[0,\overline{\gamma}/\underline{\beta}\right]\times [\mu_0,\infty).
\end{cases}
\end{equation}

\begin{prop}\label{prop:uniqueness}
There is at most one bounded viscosity solution to~\eqref{eq:hjb}.
\end{prop}
\begin{proof}
Let us assume that two continuously differentiable and bounded solutions to \eqref{eq:hjb} exist, say $v^1,v^2$.

Define 
\begin{equation*}
w(x,y,t)=
\begin{cases}
\frac{1}{2}\mu_0 t+x+y+\frac{g(x)}{y-\mu_0}, &\quad x\geq 0, y>\mu_0,t\geq0,\\
\frac{1}{2}\mu_0 t+x+y, &\quad 0\leq x\leq\overline{\gamma}/\underline{\beta},y=\mu_0,t\geq0,\\
+\infty, &\quad x>\overline{\gamma}/\underline{\beta},y=\mu_0,t\geq0,
\end{cases}
\end{equation*}
where $g=g(x):[0,\infty)\to[0,\infty)$ is a nondecreasing function that vanishes on $\left[0,\overline{\gamma}/\underline{\beta}\right]$ and that is positive in $\left(\overline{\gamma}/\underline{\beta},\infty\right)$. Then, $w$ is nonnegative, lower semicontinuous, and it satisfies 
\begin{align*}
&-\partial_t w+w+\beta x y \partial_x w +(\gamma-\beta x) y \partial_y w > 0 \quad\text{in}\quad \left(\overline{\gamma}/\underline{\beta},\infty\right)\times(\mu_0,\infty)\times(0,\infty).
\end{align*}

For a given $\varepsilon\in(0,1)$, consider 
\[
m:=\sup\{v^1(x,y,t)-v^2(x,y,t)-\varepsilon w(x,y,t):x\geq0, y\geq\mu_0,t\geq0\}.
\]
We claim that $m\leq0$. Once this claim is shown, we have $v^1\leq v^2+\varepsilon w$, and letting $\varepsilon\to0$ gives $v^1\leq v^2$. A symmetric argument will imply $v^2\leq v^1$.

Suppose not, i.e., $m>0$. Say $v^1,v^2$ are bounded by $M>0$. Then, there would exist $(\hat x,\hat y,\hat t)\in\left[0,\frac{2M}{\varepsilon}\right]\times\left[\mu_0,\frac{2M}{\varepsilon}\right]\times\left[0,\frac{4M}{\mu_0\varepsilon}\right]$ such that
$$
m=(v^1-v^2-\varepsilon w)(\hat x,\hat y,\hat t).
$$
By the boundary conditions of \eqref{eq:hjb} and the choice of $w$ (especially when $x>\overline{\gamma}/\underline{\beta}$), we see that $0<\hat x\leq\overline{\gamma}/\underline{\beta},\hat y>\mu_0,\hat t>0$. Then, by the maximum principle,
\begin{equation*}
\begin{cases}
0&=\partial_x (v^1 - v^2 -\varepsilon w)(\hat x,\hat y,\hat t),\\
0&=\partial_y (v^1 - v^2 -\varepsilon w)(\hat x,\hat y,\hat t),\\
0 &=\partial_t (v^1 - v^2 -\varepsilon w)(\hat x,\hat y,\hat t).
\end{cases}
\end{equation*}
Also, note that $(\partial_x v^1 (\hat x,\hat y,\hat t))_-=((\partial_x v^2+\varepsilon \partial_x w)(\hat x,\hat y,\hat t))_- \leq (\partial_x v^2 (\hat x,\hat y,\hat t))_-$ as $\partial_x w\geq0$. Therefore, $(\hat x,\hat y,\hat t)$, we deduce that
\begin{align*}
m&=v^1-v^2 -\epsilon w\\
&=\partial_t v^1-\beta \hat x\hat y \partial_x v^1 +\hat x(\partial_x v^1)^--(\gamma-\beta \hat x)\hat y v^1 -\left(\partial_t v^2-\beta \hat x\hat y \partial_x v^2 +\hat x(\partial_x v^2)^--(\gamma-\beta \hat x)\hat y v^2\right)-\varepsilon w\\
&\leq\partial_t(v^1 -v^2) -\beta \hat x \hat y \partial_x (v^1 - v^2)-(\gamma -\beta \hat x)\hat y\partial_y (v^1 - v^2)-\varepsilon w\\
&= -\varepsilon( -\partial_t w +w+\beta \hat x \hat y \partial_x w +(\gamma -\beta \hat x)\hat y\partial_y w)\\
&< 0,
\end{align*}
which is a contradiction. We finish the proof when the bounded solutions $v^1,v^2$ are continuously differentiable.

When the viscosity solutions $v^1,v^2$ are only continuous, we can show $v^1=v^2$ by the doubling variable technique (see \cite{Evans2010textbook,Tran2021HJtheory}). The goal is again to prove $m\leq0$. For $\varepsilon,\alpha\in(0,1)$, we let
\begin{equation}\label{eq:sup}
\begin{split}
m_\alpha:= &\sup \{ v^1(x_1,y_1,t_1)-v^2(x_2,y_2,t_2)-\frac{\epsilon}{2}(w(x_1,y_1,t_1)+w(x_2,y_2,t_2))\\
&\quad-\frac{1}{2\alpha}((x_1-x_2)^2+(y_1-y_2)^2+(t_1-t_2)^2):x_1,x_2,t_1,t_2 \geq 0,y_1,y_2\geq\mu_0\}.
\end{split}
\end{equation}
If the solutions $v^1,v^2$ are bounded by $M>0$, then there exist $(x_1^\alpha,y_1^\alpha,t_1^\alpha),(x_2^\alpha,y_2^\alpha,t_2^\alpha)\in\left[0,\frac{4M}{\varepsilon}\right]\times\left[\mu_0,\frac{4M}{\varepsilon}\right]\times\left[0,\frac{8M}{\mu_0\varepsilon}\right]$ such that the supremum~\eqref{eq:sup} is achieved and finite. By passing to a subsequence of $\alpha\to0$ if necessary, we have
\begin{equation}
\begin{cases}
\hat x := \lim_{\alpha\rightarrow 0} x_1^\alpha=\lim_{\alpha\rightarrow 0} x_2^\alpha,\\
\hat y := \lim_{\alpha\rightarrow 0} y_1^\alpha=\lim_{\alpha\rightarrow 0} y_2^\alpha,\\
\hat t := \lim_{\alpha\rightarrow 0} t_1^\alpha=\lim_{\alpha\rightarrow 0} t_2^\alpha,\\
\end{cases}
\end{equation}
and $\lim_{\alpha\rightarrow 0} m_\alpha = m$. By the boundary conditions of \eqref{eq:hjb} and the choice of $w$, we see that $0<\hat x\leq\overline{\gamma}/\underline{\beta},\hat y>\mu_0,\hat t>0$ as in \cite{hynd2021eradication}; otherwise, we readily achieve our goal to prove $m\leq0$. Therefore, we may assume the case that $(\hat x,\hat y,\hat t) \in (0,\infty) \times (\mu_0,\infty) \times (0,\infty)$ and the rest follows from the same argument as in~\cite{hynd2021eradication}. 
\end{proof}

\subsection{Local semiconcavity}

In this section, we establish the local semiconcavity of the value function $u(x,y,t)$ when $\beta(t)\equiv\beta_0,\gamma(t)\equiv\gamma_0$ for all $t\geq T$ for some $T>0$ depending only on $\mu_0,\mu,\overline{\gamma}$. Here, we keep $\mu_1\in(0,\mu_0]$ as in Theorem \ref{thm:choiceofmu} and let $\mu\in(0,\mu_1]$ so that $\overline{u}=\underline{u}$ on $[0,\infty)\times[\mu_0,\infty)\times[0,\infty)$.

The observation is from the fact that the local semiconcavity property propagates from that of initial/terminal data along the time (see \cite{barron1999regularity,cannarsa2004semiconcave}). See also the recent paper \cite{han2022semiconcavity} for the global semiconcavity. We first note the semiconcavity property of the terminal data $u(x,y,T)$ when $\beta(t)\equiv\beta_0,\gamma(t)\equiv\gamma_0$ for all $t\geq T$, which is the case studied in \cite{hynd2021eradication}.

\begin{thm}\cite[Theorem 1.3]{hynd2021eradication}\label{thm:terminaldata}
Suppose that $\beta(t)\equiv\beta_0,\gamma(t)\equiv\gamma_0$ for all $t\geq T$. Then, we have $\overline{u}(x,y,t)=\underline{u}(x,y,t)$ for all $x\geq0,y\geq\mu,t\geq T$. Moreover, $u(x,y,T):=\overline{u}(x,y,T)=\underline{u}(x,y,T)$ is locally semiconcave in $(x,y)\in(0,\infty)\times(\mu,\infty)$.
\end{thm}

Now, we prove Theorem \ref{thm:semiconcavity}.

\begin{proof}[Proof of Theorem \ref{thm:semiconcavity}]
When $\mu=\mu_0$, the theorem follows from Theorem \ref{thm:terminaldata} since $T=0$. We assume the other case $\mu<\mu_0$ in the rest of the proof.

We start with the fact that for $x>0,y>\mu_0,$ we have $I^d(t)\geq y e^{-\overline{\gamma}t}> \mu_0 e^{-\overline{\gamma}t}$. Therefore, for $T:=\frac{1}{2\overline{\gamma}}\log\left(\frac{\mu_0}{\mu}\right)$, it holds that $I^d(T)\geq y\sqrt{\frac{\mu}{\mu_0}}>\mu_0$ for every $d\in\mathcal{D}$, which also implies
\begin{align}\label{distant}
\inf\left\{I^d(T):d=(x,y,t,\alpha)\in\mathcal{D},\ y>\mu_0+\delta\right\}>\mu
\end{align}
for every $\delta>0$.

Let $v(x,y,t)=u(x,y,t)+t-T$ for $x\geq0,y\geq\mu_0,t\geq0,$ and let $g(x,y)=u(x,y,T)$ for $x\geq0, y\geq\mu$. Then, by Proposition \ref{prop:dpp}, we have, for $(x,y,t)\in(0,\infty)\times(\mu_0,\infty)\times(0,T]$,
$$
v(x,y,t)=\inf_{\alpha\in\mathcal{A}}g\left(S^d(T-t),I^d(T-t)\right),
$$
where $d=(x,y,t,\alpha)\in\mathcal{D}$.

Our goal is to show the local semiconcavity of $v$, instead of that of $u$, which is sufficient to complete the proof. To this end, we fix $t\in(0,T]$, $K\subset\subset\Omega:=(0,\infty)\times(\mu_0,\infty)$ and $(x,y)\in K$. By Propositions \ref{prop:optimalcontrol}, \ref{prop:dpp}, there exists an optimal control $\alpha^{\ast}\in\mathcal{A}$ such that $v(x,y,t)=g\left(S^{d^{\ast}}(T-t),I^{d^{\ast}}(T-t)\right)$ where $d^{\ast}=(x,y,t,\alpha^{\ast})\in\mathcal{D}$.

Let $h=(h_1,h_2)$ be a vector in $\mathbb{R}^2$ whose magnitude is smaller than $\textrm{dist}(K,\partial\Omega)$. Let $z(s)=\left(S^{d^{\ast}}(s),I^{d^{\ast}}(s)\right)$ and $z_{\pm}(s)=\left(S^{d_{\pm}^{\ast}}(s),I^{d_{\pm}^{\ast}}(s)\right)$ for $s\in[0,t]$, where $d^{\ast}_{\pm}=(x\pm h_1,y\pm h_2,t,\alpha^{\ast})$. Then, it holds that 
\[
|z_+(t)-z_-(t)| \leq c|h|\quad\text{and}\quad|z_+(t)+z_-(t)-2z(t)| \leq c|h|^2,
\]
where $c$ is a constant depending only on the compact set $K\subset\subset\Omega$ and on the rates $\beta,\gamma$. Therefore,
\begin{equation*}
\begin{split}
&v(x+h_1,y+h_2,t)+v(x-h_1,y-h_2,t)-2v(x,y,t)\\
&\leq g(z_+(t))+g(z_-(t))-2g\left(\frac{z_+(t)+z_-(t)}{2}\right)+2g\left(\frac{z_+(t)+z_-(t)}{2}\right)-2g(z(t))\\
&\leq c\left(\left|\frac{z_+(t)-z_-(t)}{2}\right|^2+\left|z_+(t)+z_-(t)-2z(t)\right|\right)\\
&\leq c|h|^2.
\end{split}
\end{equation*}
Here, $c$ denotes constants that may vary line by line, but all of which depend only on the compact set $K\subset\subset\Omega$ and on the rates $\beta,\gamma$. We used \eqref{distant} and Theorem \ref{thm:terminaldata}. Note also that the local semiconcavity of $g$ in $(0,\infty)\times(\mu,\infty)$ implies the local Lipschitz regularity of $g$ in $(0,\infty)\times(\mu,\infty)$. This proves the local semiconcavity of $v(\cdot,\cdot,t)$ (or, of $u(\cdot,\cdot,t)$) in $\Omega$ for each $t\in(0,T]$.

The local semiconcavity in time also follows from a similar argument and we refer the details to~\cite{cannarsa2004semiconcave}.
\end{proof}

\section*{Acknowledgement}
The authors are grateful to Prof. Hung Vinh Tran at University of Wisconsin-Madison for introducing this problem. Yeoneung Kim is supported by RS-2023-00219980. 

\bibliographystyle{alpha}
\bibliography{Preprint}

\end{document}